\documentclass[11pt]{amsart}
\usepackage{graphicx}
\usepackage{amssymb}
\usepackage{amscd}
\usepackage{cite}
\usepackage[all]{xy}
\usepackage{url}
\newtheorem{tw}{Theorem}[section]
\newtheorem{prop}[tw]{Proposition}
\newtheorem{lem}[tw]{Lemma}

\theoremstyle{remark}

\theoremstyle{definition}

\newcommand{\Cal}[1]{\mathcal{#1}}

\newcommand{\podz}{\subseteq}

\newcommand{\eps}{\varepsilon}
\newcommand{\ro}{\varrho}

\newcommand{\kre}[1]{\overline{#1}}
\newcommand{\gen}[1]{\langle #1 \rangle}
\newcommand{\map}[3]{#1\colon #2\to #3}
\newcommand{\field}[1]{\mathbb{#1}}
\newcommand{\zz}{\field{Z}}

\newcommand{\st}{\;|\;}

\begin{document}

\numberwithin{equation}{section}
\title[The first homology group with twisted coefficients \ldots]
{The first homology group with twisted coefficients for the mapping class group of a non--orientable surface with boundary }

\author{Piotr Pawlak \hspace{1em} Micha\l\ Stukow}

\address[]{
Institute of Mathematics, Faculty of Mathematics, Physics and Informatics, \\ University of Gda\'nsk, 80-308 Gda\'nsk, Poland }

\email{piotrus.pawlak@wp.pl, trojkat@mat.ug.edu.pl}





\begin{abstract}
We determine the first homology group with coefficients in $H_1(N;\zz)$ for various mapping class groups 
of a non--orientable surface $N$ with punctures and/or boundary.
\end{abstract}

\maketitle%
 \section{Introduction}%
Let $N_{g,s}^n$ be a smooth, non--orientable, compact surface of genus $g$ with $s$ boundary components and $n$ punctures. If $s$ and/or $n$ is zero, then we omit it from the notation. If we do not want to emphasise the numbers $g,s,n$, we simply write $N$ for a surface $N_{g,s}^n$. Recall that $N_{g}$ is a connected sum of $g$ projective planes and $N_{g,s}^n$ is obtained from $N_g$ by removing $s$ open discs and specifying a set $\Sigma=\{z_1,\ldots,z_n\}$ of $n$ distinguished points in the interior of~$N$.

Let ${\textrm{Diff}}(N)$ be the group of all diffeomorphisms $\map{h}{N}{N}$ such that $h$ is the identity 
on each boundary component and $h(\Sigma)=\Sigma$. By ${\Cal{M}}(N)$ we denote the quotient group of ${\textrm{Diff}}(N)$ by
the subgroup consisting of maps isotopic to the identity, where we assume that isotopies are 
the identity on each boundary component. ${\Cal{M}}(N)$ is called the \emph{mapping class group} of $N$. 

The mapping class group ${\Cal{M}}(S_{g,s}^n)$ of an orientable surface is defined analogously, but we consider 
only orientation preserving maps. 

For any $0\leq k\leq n$, let ${\Cal{PM}}^k(N)$ be the subgroup of ${\Cal{M}}(N)$ consisting of elements which fix $\Sigma$ pointwise and preserve a local orientation around the punctures $\{z_1,\ldots,z_k\}$. For $k=0$,
we obtain so--called \emph{pure mapping class group ${\Cal{PM}}(N)$}, and for $k=n$ we get the group 
${\Cal{PM}^+}(N)$ consisting of maps that preserve local orientation around all the punctures.

\subsection{Background}
Homological computations play a prominent role in the theory of mapping class groups. In the orientable case,
Mumford \cite{MumfordAbel} observed that $H_1({\Cal{M}}(S_g))$ is a quotient of $\zz_{10}$. Then Birman 
\cite{Bir1,Bir1Err} showed that if $g\geq 3$, then $H_1({\Cal{M}}(S_g))$ is a quotient of $\zz_2$, and
Powell \cite{Powell} showed that in fact $H_1({\Cal{M}}(S_g))$ is trivial if $g\geq 3$. As for
higher homology groups, Harer \cite{Harer-stab,Harer3} computed $H_i({\Cal{M}}(S_g))$ for $i=2,3$ 
and Madsen and Weiss \cite{MadWeiss} determined the rational cohomology ring of the stable mapping
class group.

In the non--orientable case, Korkmaz \cite{Kork-non1,Kork-non} computed $H_1({\Cal{M}}(N_g))$ for a 
closed surface $N_g$ (possibly with marked points). This computation was later \cite{Stukow_SurBg} 
extended to the case of a surface with boundary. As for higher homology groups, Wahl 
\cite{Wahl_stab} identified the stable rational cohomology of ${\Cal{M}}(N)$ and Randal-Williams \cite{RandalHomNon} (among other results) extended this identification to $\zz_2$ coefficients.

As for twisted coefficients, Morita in a series of papers \cite{MoritaJacFou,MoritaInv93,MoritaJacProc} obtained several fundamental results, in particual he proved that
\[\begin{aligned}
H_1({\Cal{M}}(S_g);H_1(S_g;\zz))&\cong \zz_{2g-2},\quad \text{for $g\geq 2$},\\
H^1({\Cal{M}}(S_g);H^1(S_g;\zz))&\cong 0,\quad \text{for $g\geq 1$},\\
H^1({\Cal{M}}(S_{g}^1);H^1(S_g^1;\zz))&\cong \zz,\quad \text{for $g\geq 2$},\\
H^1({\Cal{M}}(S_{g,1});\Lambda^3 H_1(S_{g};\zz))&\cong \zz\oplus Z,\quad \text{for $g\geq 3$}.
  \end{aligned}
\]
We showed in \cite{Stukow_homolTopApp} that if  $N_{g,s}$ is a non--orientable surface of genus $g\geq 3$ with $s\leq 1$ boundary components, then
\begin{equation}H_1({\Cal{M}}(N_{g,s});H_1(N_{g,s};\zz))\cong \begin{cases}
                                                 \zz_2\oplus\zz_2\oplus\zz_2 &\text{if $g\in\{3,4,5,6\}$,}\\
                                                 \zz_2\oplus\zz_2&\text{if $g\geq 7$.}
                                                \end{cases}\label{tw:hom:twist}
\end{equation}
There are also similar computations for the hyperelliptic mapping class groups ${\Cal{M}}^h(S_g)$. Tanaka
\cite{Tanaka} showed that $H_1({\Cal{M}}^h(S_g);H_1(S_g;\zz))\cong \zz_{2}$ for $g\geq 2$, and in the non--orientable
case we showed in \cite{Stukow_HiperOsaka} that
\[H_1({\Cal{M}}^h(N_g);H_1(N_g;\zz))\cong \zz_{2}\oplus \zz_2\oplus\zz_2,\quad \text{for $g\geq 3$.}\]
There is also a lot of interesting results concerning the stable twisted (co)homology groups of mapping class groups -- see \cite{SoulieTwsited,Kawazumi2,RandalWhal,Looijenga96} and references there.
\subsection{Main results}
The purpose of this paper is to extendthe the formula \eqref{tw:hom:twist} to the case of surfaces with punctures and/or boundary. We prove the following theorems.
\begin{tw}\label{MainThm1}
If  $N_{g,s}^n$ is a non--orientable surface of genus $g\geq 3$ with $s$ boundary components and $n$ punctures, then
\[H_1({\Cal{PM}^k}(N_{g,s}^n);H_1(N_{g,s}^n;\zz))\cong \begin{cases}
                                                 \zz_2^{3+n} &\text{if $g=3$ and $s=k=0$,}\\
                                                 \zz_2^{1+n+k} &\text{if $g=3$, $s=0$ and $k>0$,}\\
                                                 \zz_2^{n+3s+k} &\text{if $g=3$ and $s>0$,}\\
                                                 \zz_2^{3+n-k}&\text{if $g=4$ and $s=0$,}\\
                                                 \zz_2^{2+n+s-k}&\text{if $g=4$ and $s>0$,}\\
                                                 \zz_2^{3+n-k}&\text{if $g=5$ or $g=6$,}\\
                                                 \zz_2^{2+n-k}&\text{if $g\geq 7$.}
                                                \end{cases}
\] 
\end{tw}
\begin{tw}\label{MainThm2}
If  $N_{g,s}^n$ is a non--orientable surface of genus $g\geq 3$ with $s$ boundary components and $n\geq 2$ punctures, then
\[H_1({\Cal{M}}(N_{g,s}^n);H_1(N_{g,s}^n;\zz))\cong \begin{cases}
                                                \zz_2^{5}&\text{if $g\in\{3,4\}$ and $s=0$,}\\
                                                 \zz_2^{3s+2} &\text{if $g=3$ and $s>0$,}\\
                                                 \zz_2^{4+s}&\text{if $g=4$ and $s>0$,}\\
                                                 \zz_2^{5}&\text{if $g=5$ or $g=6$,}\\
                                                 \zz_2^{4}&\text{if $g\geq 7$.}
                                                \end{cases}
\] 
\end{tw}
Note that we obtained the formula (1.1) from the full presentation for the mapping class group ${\Cal{M}}(N_{g,s})$, where $g+s\geq 3$ and $s\in\{0,1\}$, obtained by Paris and Szepietowski \cite{SzepParis}. However, we do not have full presentations for the groups ${\Cal{PM}^k}(N_{g,s}^n)$ and ${\Cal{M}}(N_{g,s}^n)$, which makes our computation less straightforward.

The starting point for this computation is a simplification of known generating sets for the groups ${\Cal{PM}^k}(N_{g,s}^n)$ and ${\Cal{M}}(N_{g,s}^n)$ -- see Theorems \ref{tw:gen:pure:plus}, \ref{tw:gen:pure} and \ref{tw:gen:k} in Section \ref{sec:gen}. Then, in Sections \ref{sec:above}, \ref{secK:above} and \ref{secM:above} we perform a detailed analysis of possible relations between these generators in order to obtain a minimal set of generators for the first homology group -- see Propositions~\ref{h1:g3}, \ref{h1:g4567}, \ref{h1v:g4567}, \ref{h1v:g3} and \ref{prop:h1:vs}. The proofs that these sets of generators are indeed linearly independent occupy Sections \ref{sec:pmk:below} and \ref{sec:pm:below}. One essential ingredient in these two sections is our recent computation \cite{PawlakStukowHomoPuncturedG3} of the homology group
\[H_1({\Cal{PM}^+}(N_{3}^2);H_1(N_{3}^2;\zz))\cong \zz_2^{5}.\]
Section \ref{sec:action} is devoted to the technical details of the action of the mapping class group ${\Cal{M}}(N)$ on the first homology group $H_1(N;\zz)$. This analysis is continued in Section \ref{sec:kernel}, where we set up a technical background for the computations of the twisted first homology group of various mapping class groups -- see Propositions \ref{prop:kernel:1}, \ref{prop:kernel:2} and \ref{prop:kernel:3}.
 \section{Preliminaries}
\subsection{Non--orientable surfaces}
Represent the surface $N_{g,s}^n$ as a sphere with $g$ crosscaps $\mu_1,\ldots,\mu_g$, $n$ marked points $z_1,\ldots,z_n$, and $s$ boundary components (Figure \ref{r01}). Let 
\[\alpha_1,\ldots,\alpha_{g-1}, \beta_1,\ldots,\beta_{\left\lfloor\frac{g-2}{2}\right\rfloor},
\kre{\beta}_0,\ldots,\kre{\beta}_{\left\lfloor\frac{g-2}{2}\right\rfloor},
\delta_1,\ldots,\delta_s,\eps_1,\ldots,\eps_{s+n}\]
be two--sided circles indicated in Figures \ref{r01}, \ref{r02} and \ref{r04}. 
\begin{figure}[h]
\begin{center}
\includegraphics[width=0.84\textwidth]{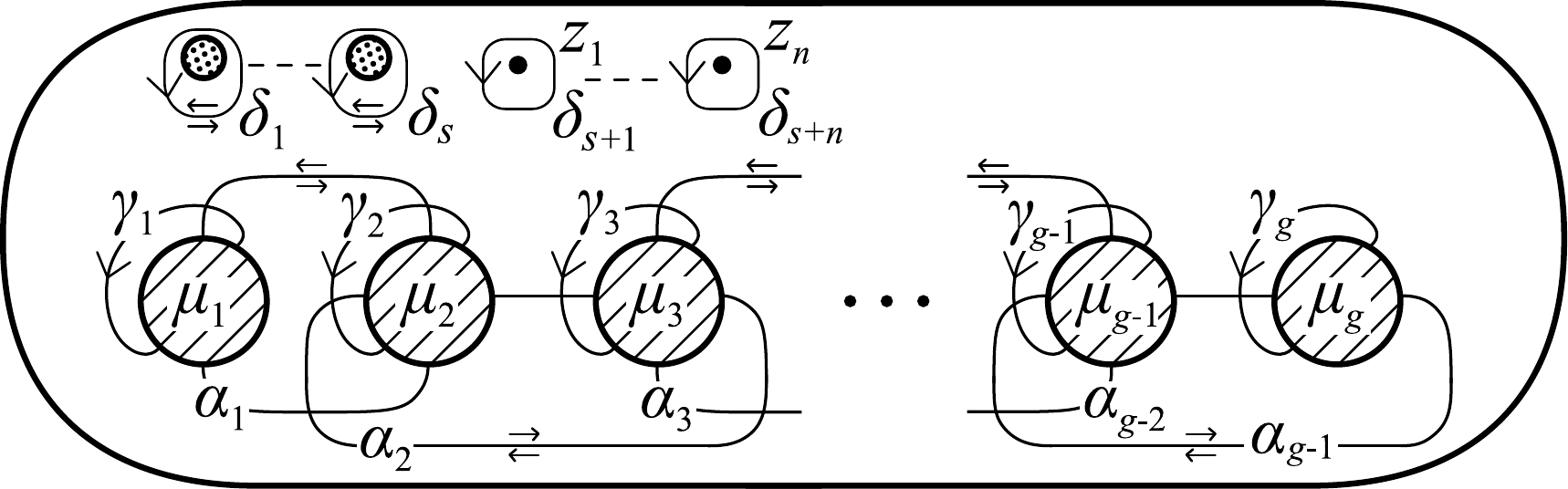}
\caption{Surface $N_{g,s}^n$ as a sphere with crosscaps.}\label{r01} %
\end{center}
\end{figure}
\begin{figure}[h]
\begin{center}
\includegraphics[width=0.95\textwidth]{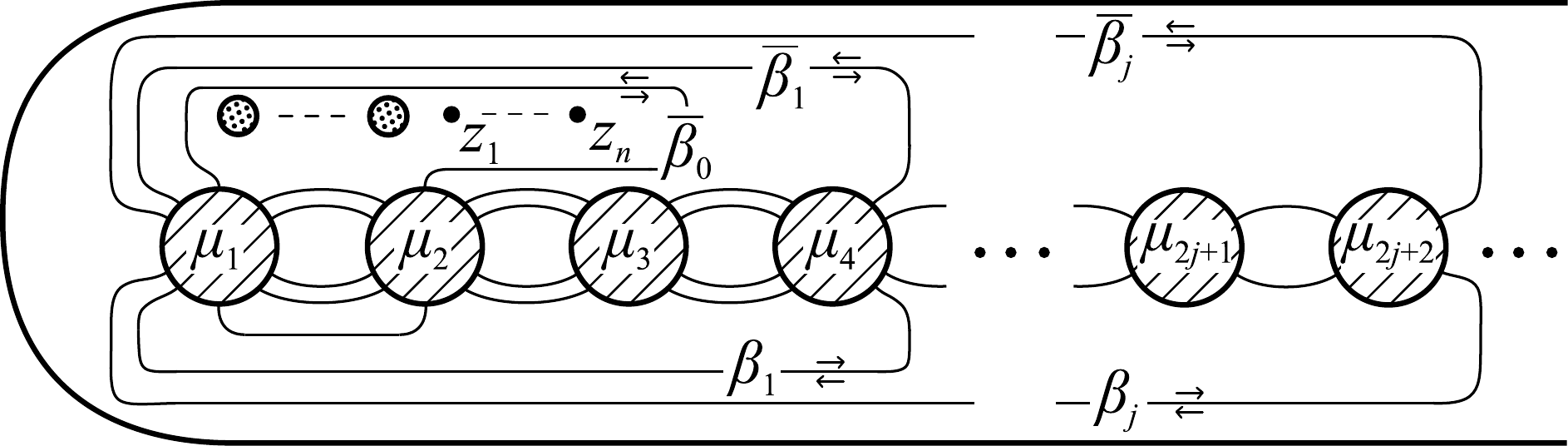}
\caption{Circles $\beta_1,\beta_2,\ldots,\beta_{\left\lfloor\frac{g-2}{2}\right\rfloor}$ and $\kre{\beta}_0,\kre{\beta}_1,\ldots,\kre{\beta}_{\left\lfloor\frac{g-2}{2}\right\rfloor}$.}\label{r02} %
\end{center}
\end{figure}
\begin{figure}[h]
\begin{center}
\includegraphics[width=0.84\textwidth]{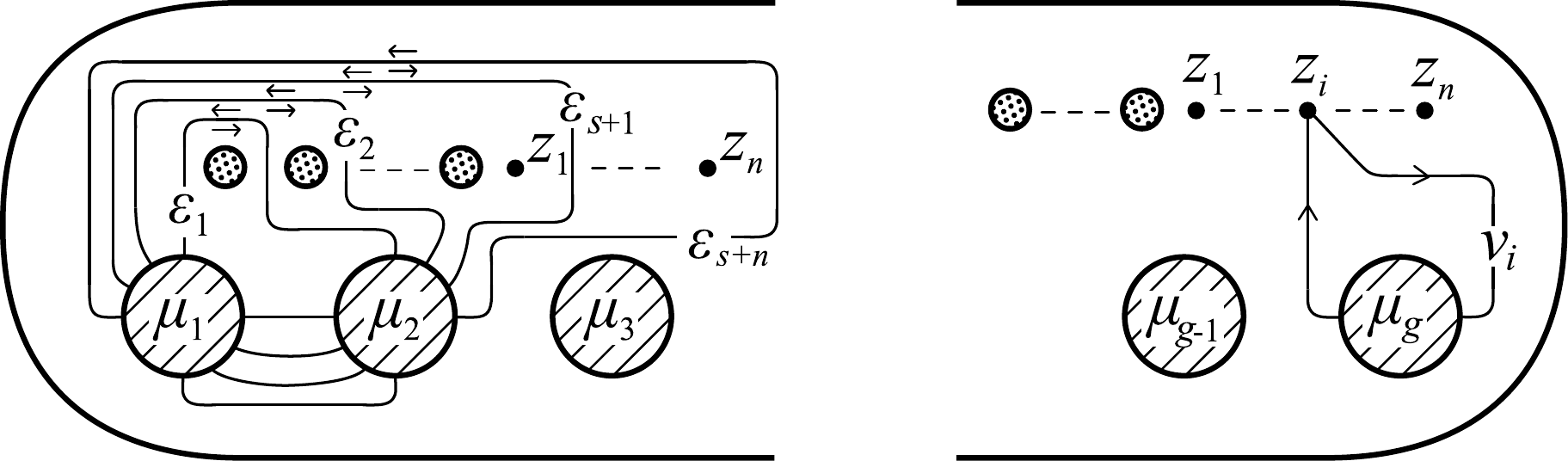}
\caption{Circles $\eps_1,\eps_2,\ldots,\eps_{s+n}$ and paths $\nu_1,\ldots,\nu_n$.}\label{r04} %
\end{center}
\end{figure}
Small arrows in these figures indicate directions of Dehn twists 
\[a_1,\ldots,a_{g-1}, b_1,\ldots,b_{\left\lfloor\frac{g-2}{2}\right\rfloor},
\kre{b}_0,\ldots,\kre{b}_{\left\lfloor\frac{g-2}{2}\right\rfloor},
d_1,\ldots,d_s,e_1,\ldots,e_{s+n}
\] associated with these circles. We also define: $\eps_0=\alpha_1$ and $e_0=a_1$.

For any two consecutive crosscaps $\mu_i,\mu_{i+1}$ we define a \emph{crosscap transposition} $u_i$ 
to be the map which interchanges these two crosscaps (see Figure \ref{r03}). Similarly, for any two consecutive punctures $z_i,z_{i+1}$ we define \emph{elementary braid} $s_i$ (Figure \ref{r03}).
\begin{figure}[h]
\begin{center}
\includegraphics[width=0.8\textwidth]{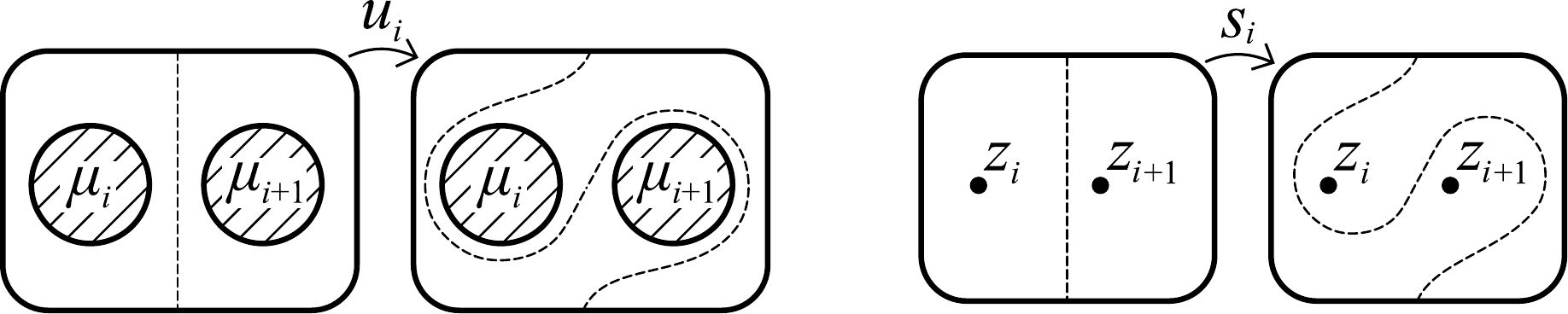}
\caption{Crosscap transposition $u_i$ and elementary braid $s_i$.}\label{r03} %
\end{center}
\end{figure}

Finally, let $v_i$, for $i=1,\ldots,n$ be the puncture slide of $z_i$ along the path $\nu_i$ indicated in Figure \ref{r04}.

\subsection{Homology of groups}
Let us briefly review how to compute the first homology of a group with twisted coefficients -- for more details see Section~5 of \cite{Stukow_HiperOsaka} and references there. 

For a given group $G$ and $G$-module $M$ (that is $\zz G$-module) we define $C_2(G)$ and $C_1(G)$ as the free $G$-modules generated
respectively by symbols $[h_1|h_2]$ and $[h_1]$, where $h_i\in G$. We define also $C_0(G)$ as the free $G$-module generated by the
empty bracket $[\cdot]$. Then the first homology group $H_1(G;M)$ is the first homology group of the complex
\[\xymatrix@C=3pc@R=3pc{C_2(G)\otimes_G M\ar[r]^{\ \partial_2\otimes_G {\rm id}}&C_1(G)\otimes_G M
\ar[r]^{\ \partial_1\otimes_G {\rm id}}&C_0(G)\otimes_G M},\]
where 
\[
\begin{aligned}
\partial_2([h_1|h_2])&=h_1[h_2]-[h_1h_2]+[h_1],\\
   \partial_1([h])&=h[\cdot]-[\cdot].
  \end{aligned}
\]
For simplicity, we write $\otimes_G=\otimes$ and $\partial\otimes {\rm id}=\kre{\partial}$ henceforth.

If the group $G$ has a presentation $G=\langle X\,|\,R\rangle$, then we have an excat sequence
\begin{equation}\label{ex:sq:free}
\xymatrix{1\ar[r]&N(R)\ar[r]&F(X)\ar[r]&G\ar[r]&1},
\end{equation}
where $F(X)$ is the free group generated by elements of $X$ (generators) and $N(R)$ is the normal closure in $F(X)$ of the set of relations $R$. Sequence \eqref{ex:sq:free} gives the following excat seqence (see for example \cite{EckmannHom} and references there)
\[\label{ex:sq:f}
\xymatrix{N(R)/[N(R),N(R)]\otimes_G M\ar[r]^-{i}&H_1(F(X);M)\ar[r]&H_1(G;M)\ar[r]&0},
\]
where $N(R)/[N(R),N(R)]$ is the abelianization of $N(R)$. Hence, we can identify
\[H_1(G;M)=H_1(F(X);M)/\textrm{Im}(i).\]
Let us now describe how to use the above identification in practice.
Let
\[\gen{\kre{X}}=\gen{[x]\otimes m\st x\in X, m\in M}\podz C_1(F(X))\otimes M,\]
then $H_1(G;M)$ is a quotient of  $\gen{\kre{X}}\cap \ker\kre{\partial}_1$.

The kernel of this quotient corresponds to relations in $G$
(that is elements of $R$). To be more precise, if
$r\in R$ has the form $x_1\cdots x_k=y_1\cdots y_n$ and $m\in M$, then 
$i(r)\in H_1(F(X);M)$ is equal to
\begin{equation}
 \kre{r}\otimes m\!:\ \sum_{i=1}^{k}x_1\cdots x_{i-1}[x_i]\otimes m-\sum_{i=1}^{n}y_1\cdots y_{i-1}[y_i]\otimes m.\label{eq_rew_rel}
\end{equation}
Then 
\[H_1(G;M)=\gen{\kre{X}}\cap \ker\kre{\partial}_1/\gen{\kre{R}},\]
where 
\[\kre{R}=\{\kre{r}\otimes m\st r\in R,m\in M\}.\]
 \section{Action of ${\Cal{M}}(N_{g,s}^n)$ on $H_1(N_{g,s}^n;\zz)$}\label{sec:action}
Let $\gamma_1,\ldots,\gamma_{g},\delta_1,\ldots,\delta_{s+n}$ be circles indicated in Figure \ref{r01}. Note that $\gamma_1,\ldots,\gamma_{g}$ are one--sided, $\delta_1,\ldots,\delta_{s+n}$ are two--sided and the $\zz$-module $H_1(N_{g,s}^n;\zz)$ is generated by homology classes $[\gamma_1],\ldots,[\gamma_g],[\delta_1],\ldots,[\delta_{s+n-1}]$. These generators are free provided $s+n>0$. In abuse of notation we will not distinguish between the curves $\gamma_1,\ldots,\gamma_{g},\delta_1,\ldots,\delta_{s+n}$ and their cycle classes.

The mapping class group ${\Cal{M}}(N_{g,s}^n)$ acts on $H_1(N_{g,s}^n;\zz)$, hence we have a representation
\[\map{\psi}{{\Cal{M}}(N_{g,s}^n)}{\textrm{Aut}(H_1(N_{g,s}^n;\zz))}. \]
It is straightforward to check that
\begin{equation}\begin{aligned}\label{eq:psi}
   \psi(a_j)&=I_{j-1}\oplus \begin{bmatrix}
0&1\\
-1&2\end{bmatrix}\oplus I_{g-j-1}\oplus I_{s+n-1},\\
\psi(a_j^{-1})&=I_{j-1}\oplus \begin{bmatrix}
2&-1\\
1&0
    \end{bmatrix}\oplus I_{g-j-1}\oplus I_{s+n-1},\\
    \psi(u_j)=\psi(u_j^{-1})&=I_{j-1}\oplus \begin{bmatrix}
0&1\\
1&0
    \end{bmatrix}\oplus I_{g-j-1}\oplus I_{s+n-1},
\end{aligned}\end{equation} 
    \begin{equation}\begin{aligned}
    \psi(b_1)&=\left(I_g+
    \begin{bmatrix}-1&1&-1&1\\
     -1&1&-1&1\\
     -1&1&-1&1\\
     -1&1&-1&1
    \end{bmatrix}\oplus 0_{g-4}\right)\oplus I_{s+n-1},\\
    \psi(b_1^{-1})&=\left(I_g+
    \begin{bmatrix}1&-1&1&-1\\
     1&-1&1&-1\\
     1&-1&1&-1\\
     1&-1&1&-1
    \end{bmatrix}\oplus 0_{g-4}\right)\oplus I_{s+n-1},\\
    \end{aligned}\end{equation}    
    \begin{equation}\begin{aligned}
    \psi(e_j)(\xi)&=\begin{cases}
                        -\gamma_2-\delta_1-\delta_2+\ldots-\delta_j&\text{if $\xi=\gamma_1$ and $j<s+n$,}\\
                        \gamma_1+2\gamma_2+\delta_1+\delta_2+\ldots+\delta_j&\text{if $\xi=\gamma_2$ and $j<s+n$,}\\
                       \xi&\text{if $\xi\neq \gamma_1$, $\xi\neq \gamma_2$ and $j<s+n$,}
                      \end{cases}\\
                      \psi(e_j^{-1})(\xi)&=\begin{cases}
                        2\gamma_1+\gamma_2+\delta_1+\delta_2+\ldots+\delta_j&\text{if $\xi=\gamma_1$ and $j<s+n$,}\\
                        -\gamma_1-\delta_1-\delta_2+\ldots-\delta_j&\text{if $\xi=\gamma_2$ and $j<s+n$,}\\
                       \xi&\text{if $\xi\neq \gamma_1$, $\xi\neq \gamma_2$ and $j<s+n$,}
                      \end{cases}\\
    \psi(d_j)&=\psi(d_j^{-1})=I_{g}\oplus I_{s+n-1},\\
    \end{aligned}\end{equation}    
    \begin{equation}\begin{aligned}
    \psi(s_j)=\psi(s_j^{-1})&=I_g\oplus  I_{s+j-1}\oplus \begin{bmatrix}
0&1\\
1&0
    \end{bmatrix}\oplus I_{n-j-2}, \text{ if $j<n-1$,}\\
    \psi(s_{n-1})(\xi)=\psi(s_{n-1}^{-1})(\xi)&=\begin{cases}
                -(2\gamma_1+\ldots+2\gamma_g-\delta_1+\ldots+\delta_{s+n-1})&\text{if $\xi=\delta_{s+n-1}$,}\\
                \xi&\text{if $\xi\neq\delta_{s+n-1}$,}
                                                \end{cases}
                                   \end{aligned}\end{equation}    
    \begin{equation}\label{eq:vi}\begin{aligned}             
\psi(v_j)(\xi)&=\psi(v_j^{-1})(\xi)=\begin{cases}
                -\delta_{s+j}&\text{if $\xi=\delta_{s+j}$ and $j<n$,}\\
                \gamma_g+\delta_{s+j}&\text{if $\xi=\gamma_g$ and $j<n$,}\\
                \xi&\text{if $\xi\neq \gamma_g$ and $\xi\neq \delta_{s+j}$ and $j<n$,}
               \end{cases}\\
               \psi(v_n)(\xi)&=\psi(v_n^{-1})(\xi)=\begin{cases}
                \gamma_g-(2\gamma_1+\ldots+2\gamma_g+\delta_1+\ldots+\delta_{s+n-1})&\text{if $\xi=\gamma_g$,}\\
                \xi&\text{if $\xi\neq \gamma_g$,}
               \end{cases}
  \end{aligned}
\end{equation}
where $I_k$ is the identity matrix of rank $k$.
\section{Generators for the groups ${\Cal{PM}}^k(N_{g,s}^n)$ and ${\Cal{M}}(N_{g,s}^n)$}\label{sec:gen}
The main goal of this section is to obtain simple generating sets for the groups ${\Cal{PM}}^k(N_{g,s}^n)$ and ${\Cal{M}}(N_{g,s}^n)$ -- see Theorems \ref{tw:gen:pure:plus},  \ref{tw:gen:pure} and \ref{tw:gen:k} below. However, we first prove some technical lemmas.

For the rest of this section, let $a_{j,i},\widehat{b}_1,\widehat{b}_{1,i}, \widehat{d}_i$ for $j=1,2,3,4$, $i=1,\ldots,s+n$ be Dehn twists about circles $\alpha_{j,i},\widehat{\beta}_1,\widehat{\beta}_{1,i},\widehat{\delta}_i$ shown in Figure \ref{r05}.
\begin{figure}[h]
\begin{center}
\includegraphics[width=0.99\textwidth]{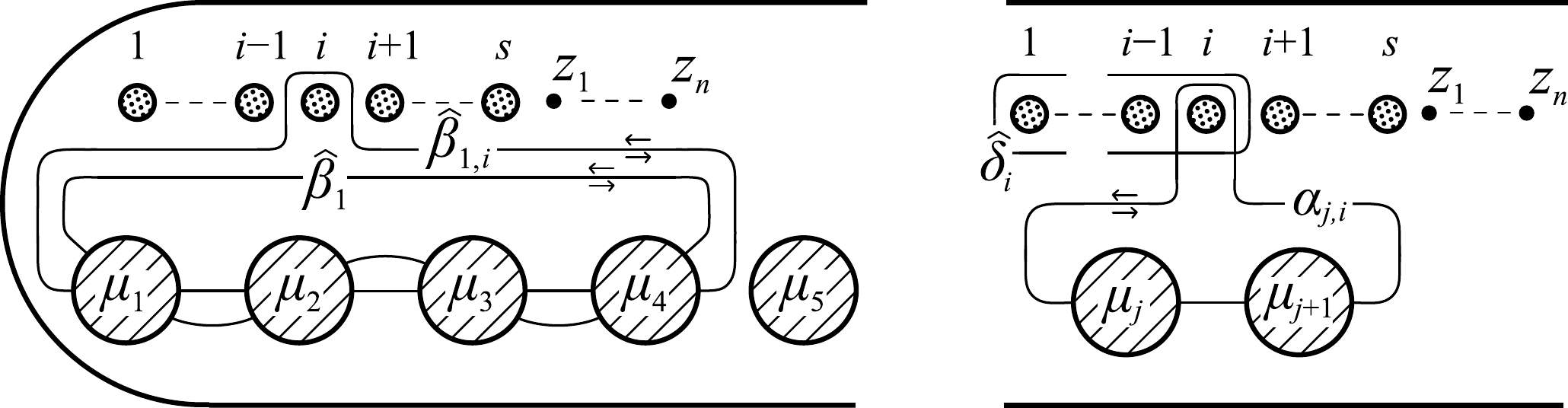}
\caption{Circles $\alpha_{j,i},\widehat{\beta}_1,\widehat{\beta}_{1,i},\widehat{\delta}_i$.}\label{r05} %
\end{center}
\end{figure}
\begin{lem}\label{lem:a:1i}
 Let $g\geq 3$ and $1\leq i\leq s+n$, then $a_{1,i}$ and $a_{2,i}$ are in the subgroup $G\leq {\Cal{M}}(N_{g,s}^n)$ generated by
 \[\left\{u_1,u_2,a_2,e_{i-1},e_{i}\right\}.\]
\end{lem}
\begin{proof}
 It is straightforward to check that
 \[\begin{aligned}
   \alpha_{2,i}&=e_i a_2(\eps_{i-1}),\\
   \alpha_{1,i}&=u_2u_1(\alpha_{2,i}).
  \end{aligned}
\]
Hence, 
\[\begin{aligned}
  a_{2,i}&=e_ia_2e_{i-1}a_2^{-1}e_{i}^{-1}\in G,\\
  a_{1,i}&=u_2u_1 a_{2,i}u_1^{-1}u_2^{-1}\in G.
  \end{aligned}
  \]
\end{proof}
\begin{lem}\label{lem:di:5}
 Let $g\geq 5$ and $1\leq i\leq s$, then $d_i$ is in the subgroup $G\leq {\Cal{M}}(N_{g,s}^n)$ generated by 
 \[\left\{u_1,u_2,u_3,u_4,a_1,a_2,a_3,a_4,e_{i-1},e_{i},b_1\right\}.\]
\end{lem}
\begin{proof}
 By Lemma \ref{lem:a:1i}, $a_{1,i},a_{2,i}\in G$. Moreover,
\[\begin{aligned}
   \alpha_{3,i}&=u_2^{-1}u_3^{-1}(\alpha_{2,i}),\\
   \alpha_{4,i}&=u_3^{-1}u_4^{-1}(\alpha_{3,i}),\\
   \widehat{\beta}_1&=u_3^{-1}u_1(\beta_1),\\
   \widehat{\beta}_{1,i}&=a_4a_{4,i}^{-1}(\widehat{\beta}_1).
  \end{aligned}
\]
This proves that 
\[a_{3,i},\widehat{b}_1,\widehat{b}_{1,i}\in G.\]
Moreover, by Lemma 6.12 of \cite{Stukow_SurBg}, there is a lantern relation
\[d_ia_1a_3\widehat{b}_{1,i}=a_{3,i}a_{1,i}\widehat{b}_1.\]
This proves that $d_i\in G$.
\end{proof}
\begin{lem}\label{lem:ds}
 Let $g\geq 3$ and $s>0$, then  $d_s$ is in the subgroup $G\leq {\Cal{M}}(N_{g,s}^n)$ generated by 
 \[\left\{u_1,\ldots,u_{g-1},a_1,\ldots,a_{g-1},e_1,\ldots,e_{s+n-1},d_1,\ldots,d_{s-1}\right\}.\]
\end{lem}
\begin{proof}
Let $H\leq G$ be the subgroup of $G$ generated by
\[\left\{u_1,\ldots,u_{g-1},a_1,\ldots,a_{g-1},e_1,\ldots,e_{s+n-1}\right\}\]
and let $d_{s+1}=d_{s+2}=\ldots=d_{s+n}=1$.

Note first that,
\[
 \eps_{s+n}=a_2a_3\ldots a_{g-1}u_{g-1}\dots u_3u_{2}(\alpha_1),
\]
and 
\begin{equation}\label{di:3:esn}
 e_{s+n}=(a_2\ldots a_{g-1}u_{g-1}\dots u_{2})a_1^{-1}(a_2\ldots a_{g-1}u_{g-1}\dots u_{2})^{-1} \in H.
\end{equation}
 We will prove by induction, that for each $k=1,2,\ldots,s+n$,
 \begin{equation}\label{di:3:ind}
  d_1d_2\ldots d_{k}\left(\widehat{d}_{k}\right)^{-1}\in H.
 \end{equation}
 Since $\widehat{d}_1=d_1$, the statement is true for $k=1$.
 
 It is straightforward to check that, for each $k=1,2,\ldots,s+n$ there is a lantern relation
 \begin{equation}\label{di:3:lantern}
  a_1e_{k+1}\widehat{d}_kd_{k+1}=e_ka_{1,k+1}\widehat{d}_{k+1}.
 \end{equation}
 Assume that 
 \[\widehat{d}_k=hd_1d_2\ldots d_k \]
 for some $h\in H$. Then,  by the formulas \eqref{di:3:esn}, \eqref{di:3:lantern} and Lemma \ref{lem:a:1i}, we have
 \[\begin{aligned}
  \widehat{d}_kd_{k+1}\widehat{d}_{k+1}^{-1}&=e_{k+1}^{-1}a_1^{-1}e_ka_{1,k+1},\\
  hd_1d_2\ldots d_k d_{k+1}\widehat{d}_{k+1}^{-1}&=e_{k+1}^{-1}a_1^{-1}e_ka_{1,k+1},\\
  d_1d_2\ldots d_k d_{k+1}\widehat{d}_{k+1}^{-1}&=h^{-1}e_{k+1}^{-1}a_1^{-1}e_ka_{1,k+1}\in H.
 \end{aligned}\]
 This completes the inductive proof of \eqref{di:3:ind}. In particular,
 \[d_1d_2\ldots d_s=d_1d_2\ldots d_{s+n}=h\widehat{d}_{s+n},\]
 for some $h\in H$. Moreover, it is straightforward to check that 
 \[\widehat{d}_{s+n}=(u_1u_2\ldots u_{g-1})^{g}\in G,\]
 hence
 \[d_s=d_{s-1}^{-1}\ldots d_2^{-1}d_1^{-1}h\widehat{d}_{s+n}\in G.\]
\end{proof}
\begin{tw}\label{tw:gen:pure:plus}
Let $g\geq 3$.
 Then the mapping class group ${\Cal{PM}}^+(N_{g,s}^n)={\Cal{PM}}^n(N_{g,s}^n)$ is generated by
 \[\{a_1,\ldots,a_{g-1},u_{1},e_1,\ldots,e_{s+n-1}\}\]
 and additionally 
 \[\begin{cases}
   \{d_1,\ldots,d_{s-1}\}&\text{if $g=3$,}\\
   \{b_1,d_1,\ldots,d_{s-1}\}&\text{if $g=4$,}\\
   \{b_1\}&\text{if $g\geq 5$.}
   \end{cases}
\]
\end{tw}
\begin{proof}
Let $G$ be the subgroup of ${\Cal{PM}}^+(N_{g,s}^n)$ generated by elements specified in the statement of the theorem. By Theorem 4.1 of \cite{Stukow_HomTw}, ${\Cal{PM}}^+(N_{g,s}^n)$ is generated by the crosscap slide $y=a_{1}u_{1}$ and $2g+n+2s-4$ twists:
 \[\left\{a_1,\ldots,a_{g-1},b_1,\ldots,b_{\left\lfloor\frac{g-2}{2}\right\rfloor},
 \kre{b}_0,\ldots,\kre{b}_{\left\lfloor\frac{g-2}{2}\right\rfloor},
 e_1,\ldots,e_{s+n-1},d_1,\ldots,d_s\right\}\]
 (note that $\kre{b}_{\left\lfloor\frac{g-2}{2}\right\rfloor}=b_{\left\lfloor\frac{g-2}{2}\right\rfloor}^{-1}$ if $g$ is even). It is enough to show that all these generators are in $G$.
 
 By Theorem 3.1 of \cite{SzepParis}, 
 \[\begin{aligned}
 b_{i+1}=(b_{i-1}a_{2i}a_{2i+1}a_{2i+2}a_{2i+3}b_i)^5&(b_{i-1}a_{2i}a_{2i+1}a_{2i+2}a_{2i+3})^{-6},\\
 &\text{for $i=1,\ldots,{\left\lfloor\frac{g-2}{2}\right\rfloor}-1$ and $b_0=a_0$}.
   \end{aligned}\]
   Hence $b_i\in G$, for $i=2,\ldots,{\left\lfloor\frac{g-2}{2}\right\rfloor}$.
 
 By Lemma 3.8 of \cite{SzepParis}, 
 \[u_{i+1}=a_{i}a_{i+1}u_{i}^{-1}a_{i+1}^{-1}a_{i}^{-1},\text{ for $i=1,\ldots,g-2$},\]
 hence $u_2,u_3,\ldots,u_{g-1}\in G$. Now it is straightforward to check, that 
 \[a_{2i+2}\cdots a_{g-2}a_{g-1}u_{g-1}\cdots u_{2i+3}u_{2i+2}(\beta_i)=\kre{\beta}_i,
 \text{ for $i=0,\ldots,{\left\lfloor\frac{g-2}{2}\right\rfloor}$}.\]
This shows, that $\kre{b}_i$ is conjugate to $b_i$ by an element of $G$. Hence $\kre{b}_i\in G$, for $i=0,\ldots,{\left\lfloor\frac{g-2}{2}\right\rfloor}$. This together with Lemma \ref{lem:ds}  complete the proof if $g<5$.

Finally, if $g\geq 5$, then Lemma \ref{lem:di:5} implies that $d_i\in G$, for $i=1,\ldots,s$. 
\end{proof}
\begin{tw} \label{tw:gen:pure}
 Let $g\geq 3$ and $0\leq k\leq n$. 
 Then the mapping class group ${\Cal{PM}}^k(N_{g,s}^n)$ is generated by ${\Cal{PM}}^+(N_{g,s}^n)$ and $(n-k)$ puncture slides
 \[\{v_{k+1},\ldots,v_n\}.\]
\end{tw}
\begin{proof}
The statement follows from the short exact sequence 
\[\begin{CD}1@>>>{\Cal{PM}}^+(N_{g,s}^n) @>i>> {\Cal{PM}}^k(N_{g,s}^n)@>p>> \zz_2^{n-k}@>>>1\end{CD}\]
and the fact that $\left\{p(v_{k+1}),\ldots,p(v_n)\right\}$ generate $\zz_2^{n-k}$.
\end{proof}
\begin{tw}\label{tw:gen:k}
 Let $g\geq 3$ and $n\geq 2$. Then the mapping class group ${\Cal{M}}(N_{g,s}^n)$ is generated by ${\Cal{PM}}^+(N_{g,s}^n)$ and
 \[\{v_{n},s_1,\ldots,s_{n-1}\}.\]
\end{tw}
\begin{proof}
By Theorem \ref{tw:gen:pure}, the pure mapping class group ${\Cal{PM}}(N_{g,s}^n)={\Cal{PM}}^0(N_{g,s}^n)$ is
generated by ${\Cal{PM}}^+(N_{g,s}^n)$ and $\{v_1,\ldots,v_n\}$. Moreover, we have the short exact sequence 
\[\begin{CD}1@>>>{\Cal{PM}}(N_{g,s}^n) @>i>> {\Cal{M}}(N_{g,s}^n)@>p>> S_n@>>>1\end{CD},\]
where $S_n$ is the symmetric group on $n$ letters. Now the statement follows from the fact that 
$p(s_1),\ldots,p(s_{n-1})$ generate $S_n$ and the relation
\[v_{i-1}=s_{j-1}^{-1}v_{j}s_{j-1},\quad \text{ for $i=2,\ldots,n$}.\]
\end{proof}
For further reference, let us prove that
\begin{prop} \label{prop:rel:sj:ej}
 Let $g\geq 3$, $n\geq 2$ and $1 \leq j\leq n-1$. Then
 \[e_{s+j-1}e_{s+j+1}s_j=e_{s+j}s_j^3e_{s+j}.\]
\end{prop}
\begin{proof}
 It is straightforward to check that $\eps_{s+j+1}$ bounds in $N$ a disk with three holes: $\eps_{s+j-1},\delta_{j},\delta_{j+1}$. This implies that there is a lantern relation of the form
 \[\begin{aligned}
    e_{s+j+1}e_{s+j-1}d_{j}d_{j+1}&=e_{s+j}s_j^2\left(s_je_{s+j}s_j^{-1}\right),\\
    e_{s+j+1}e_{s+j-1}&=e_{s+j}s_j^3e_{s+j}s_j^{-1},\\
    e_{s+j+1}e_{s+j-1}s_j&=e_{s+j}s_j^3e_{s+j}.
   \end{aligned}
\]
\end{proof}
\section{Computing $\gen{\kre{X}}\cap \ker\kre{\partial}_1$} \label{sec:kernel}
Let $G={\Cal{M}}(N_{g,s}^n)$,  $M=H_1(N_{g,s}^n;\zz)$ and assume that $s+n>0$.  Let
\[\xi_i=\begin{cases}
         \gamma_i&\text{for $i=1,\ldots,g$,}\\
         \delta_{i-g}&\text{for $i=g+1,\ldots,g+s+n-1$.}
        \end{cases}
\]
If $h\in G$, then
\[\kre{\partial}_1([h]\otimes\xi_i)=(h-1)[\cdot]\otimes\xi_i=(\psi(h)^{-1}-I_g)\xi_i,\]
where we identified $C_0(G)\otimes M$ with $M$ by the map $[\cdot]\otimes m\mapsto m$.

Let us denote 
\[[a_j]\otimes \xi_i,\ [u_{j}]\otimes\xi_i,\ [b_1]\otimes\xi_i,\ [e_j]\otimes\xi_i,\
[d_j]\otimes\xi_i,\ [s_j]\otimes\xi_i,\ [v_j]\otimes\xi_i\] 
respectively by 
\[a_{j,i},\ u_{j,i},\ b_{1,i},\ e_{j,i},\ d_{j,i},\ s_{j,i},\ v_{j,i},\]
where $i=1,\ldots,g+s+n-1$.

Using formulas \eqref{eq:psi}--\eqref{eq:vi}, we obtain
\begin{equation}\begin{aligned}\label{eq:partial}
\kre{\partial}_1(a_{j,i})&=\begin{cases}
                           \gamma_j+\gamma_{j+1}&\text{if $i=j$,}\\
                           -\gamma_j-\gamma_{j+1}&\text{if $i=j+1$,}\\
                           0&\text{otherwise,}
                          \end{cases}
\\
\kre{\partial}_1(u_{j,i})&=\begin{cases}
                           -\gamma_j+\gamma_{j+1}&\text{if $i=j$,}\\
                           \gamma_j-\gamma_{j+1}&\text{if $i=j+1$,}\\
                           0&\text{otherwise,}
                          \end{cases}
\\
\kre{\partial}_1(b_{1,i})&=\begin{cases}
                           \gamma_1+\gamma_2+\gamma_3+\gamma_{4}&\text{if $i=1,3$,}\\
                           -\gamma_1-\gamma_2-\gamma_3-\gamma_{4}&\text{if $i=2,4$,}\\
                           0&\text{otherwise,}
                          \end{cases}
\end{aligned}\end{equation}
\begin{equation}\begin{aligned}
\kre{\partial}_1(e_{j,i})&=\begin{cases}
            \gamma_1+\gamma_2+(\delta_1+\ldots+\delta_j)&\text{if $i=1$,}\\
                           -\gamma_1-\gamma_2-(\delta_1+\ldots+\delta_j)&\text{if $i=2$,}\\
                           0&\text{otherwise,}
                          \end{cases}\\
    \kre{\partial}_1(d_{j,i})&=0,\\
    \text{if $j<n-1$, then }\kre{\partial}_1(s_{j,i})&=\begin{cases}
            -\delta_{s+j}+\delta_{s+j+1}&\text{if $i=g+s+j$,}\\
                           \delta_{s+j}-\delta_{s+j+1}&\text{if $i=g+s+j+1$,}\\
                           0&\text{otherwise,}
                          \end{cases}\\
                          \kre{\partial}_1(s_{n-1,i})&=\begin{cases}
            -(2\gamma_1+\ldots+2\gamma_g)\\
            \quad -(\delta_1+\ldots+\delta_{s+n-1})\\
            \quad -\delta_{s+n-1}&\text{if $i=g+s+n-1$,}\\
                           0&\text{otherwise,}
                          \end{cases}\\
\text{if $j<n$, then } \kre{\partial}_1(v_{j,i})&=\begin{cases}
                                                      \delta_{s+j}&\text{if $i=g$,}\\
                                                    -2\delta_{s+j}&\text{if $i=g+s+j$,}\\
                                                    0&\text{otherwise,}
                                                     \end{cases}\\
\kre{\partial}_1(v_{n,i})&=\begin{cases}
            -(2\gamma_1+\ldots+2\gamma_g)\\
            \quad -(\delta_1+\ldots+\delta_{s+n-1})&\text{if $i=g$,}\\
                           0&\text{otherwise.}
                          \end{cases}\\
                 \end{aligned}\end{equation}                
The above formulas show that all of the following elements are contained in $\ker\kre{\partial}_1$
 \begin{enumerate}
 \item[(K1)] $a_{j,i}$ for $j=1,\ldots,g-1$ and $i=1,\ldots,j-1,j+2,\ldots,g+s+n-1$,
\item[(K2)] $a_{j,j}+a_{j,j+1}$ for $j=1,\ldots,g-1$,
\item[(K3)] $u_{1,i}$ for $i=3,\ldots,g+s+n-1$,
\item[(K4)] $u_{1,1}+u_{1,2}$,
\item[(K5)] $e_{j,i}$ for $j=1,\ldots,s+n-1$ and $i=3,4,\ldots,g+s+n-1$,
\item[(K6)] $e_{j,1}+e_{j,2}$ for $j=1,\ldots,s+n-1$,
\item[(K7)] $d_{j,i}$ for $j=1,\ldots,s-1$ and $i=1,\ldots, g+s+n-1$,
\item[(K8)] $b_{1,i}$ for $i=5,\ldots,g+s+n-1$,
\item[(K9)] $b_{1,i}+b_{1,1}$ for $i=2,4$,
\item[(K10)] $b_{1,3}-b_{1,1}$,
\item[(K11)] $b_{1,1}-a_{1,1}-a_{3,3}$.
\end{enumerate}
\begin{prop}\label{prop:kernel:1}
 Let $g\geq 3$, $s+n>0$ and $G={\Cal{PM}}^+(N_{g,s}^n)$. Then $\gen{\kre{X}}\cap \ker\kre{\partial}_1$ is the abelian group generated freely by Generators (K1)--(K6) and additionally
 \[\begin{cases}
    \text{(K7)}&\text{if $g=3$,}\\
    \text{(K7), (K8)--(K11)}&\text{if $g=4$,}\\
    \text{(K8)--(K11)}&\text{if $g\geq 5$.}
   \end{cases}
\]
\end{prop}
\begin{proof}
 By Theorem \ref{tw:gen:pure:plus}, $\gen{\kre{X}}$ is generated freely by $a_{j,i},u_{1,i},e_{j,i}$ and
 \[\begin{cases}
    d_{j,i}&\text{if $g=3$,}\\
    b_{1,i},d_{j,i}&\text{if $g=4$,}\\
    b_{1,i}&\text{if $g\geq 5$.}
   \end{cases}
\]
Suppose that $h\in\gen{\kre{X}}\cap\ker\kre{\partial}_1$. We will show that $h$ can be uniquely expressed as a linear combination of generators specified in the statement of the proposition.

We decompose $h$ as follows:
\begin{itemize}
 \item $h=h_0=h_1+h_2$, where $h_1$ is a combination of Generators (K1)--(K2) and $h_2$ does not contain $a_{j,i}$ with $i\neq j$;
 \item $h_2=h_3+h_4$, where $h_3$ is a combination of Generators (K3)--(K4) and $h_4$ does not contain $u_{1,i}$ with $i\neq 1$;
 \item $h_4=h_5+h_6$, where $h_5$ is a combination of Generators (K5)--(K6) and $h_6$ does not contain $e_{j,i}$ for $i>1$;
\end{itemize}
If $g=3$ or $g=4$, we decompose $h_6=h_7+h_8$, where $h_7$ is a combination of Generators (K7) and $h_8$ does not contain $d_{j,i}$. If $g\geq 5$, we define $h_7=0$ and $h_8=h_6$.

If $g\geq 4$, we decompose $h_8=h_9+h_{10}$, where $h_9$ is a combination of Generators (K8)--(K11) and $h_{10}$ does not contain $b_{1,i}$. If $g=3$ we define $h_9=0$ and $h_{10}=h_8$.

Observe also that for each $k=0,\ldots,8$, $h_{k+1}$ and $h_{k+2}$ are uniquely determined by $h_k$. Element $h_{10}$ has the form \[h_{10}=\sum_{j=1}^{g-1}\alpha_ja_{j,j}+\alpha u_{1,1}+
\sum_{j=1}^{s+n-1}\beta_j e_{j,1}
\]
for some integers $\alpha,\alpha_1,\ldots,\alpha_{g-1},\beta_1,\ldots,\beta_{s+n-1}$. Hence
\[\begin{aligned}
 0=&\kre{\partial}_1(h_{10})=\alpha_1(\gamma_1+\gamma_2)+\alpha_2(\gamma_2+\gamma_3)+\ldots+\alpha_{g-1}(\gamma_{g-1}+\gamma_g)\\
 &+\alpha(-\gamma_{1}+\gamma_2)+\beta_1(\gamma_1+\gamma_2+\delta_1)+\beta_2(\gamma_1+\gamma_2+\delta_1+\delta_2)\\
 &+\ldots+\beta_{s+n-1}(\gamma_1+\gamma_2+\delta_1+\delta_2+\ldots+\delta_{s+n-1}).
\end{aligned}\]
This implies that $\beta_{s+n-1}=\ldots=\beta_2=\beta_1=0$, and then $\alpha_{g-1}=\ldots=\alpha_2=\alpha_{1}=\alpha=0$ and thus $h_{10}=0$.
\end{proof}
By an analogous argument and Propositions \ref{tw:gen:pure}, \ref{tw:gen:k}, we get
\begin{prop}\label{prop:kernel:2}
 Let $g\geq 3$, $s+n>0$, $0\leq k\leq n$ and $G={\Cal{PM}}^k(N_{g,s}^n)$. Then $\gen{\kre{X}}\cap \ker\kre{\partial}_1$ is the abelian group generated by generators specified in the statement of Proposition \ref{prop:kernel:1} and additionally
 \begin{enumerate}
  \item [(K12)] $\widetilde{v}_{j,i}$ for $k<j\leq n$ and $1\leq i\leq g+s+n-1$,
 \end{enumerate}
 where 
 \begin{equation*} 
  \widetilde{v}_{j,i}=
  \begin{cases}
  v_{j,i}&\text{if $k<j\leq n$ and $i\not\in\{g,g+s+j\}$,}\\
  v_{j,g}+e_{s+j-1,1}-e_{s+j,1}&\text{if $k<j<n$ and $i=g$,}\\
  v_{j,g+s+j}-2e_{s+j-1,1}+2e_{s+j,1}&\text{if $k<j<n$ and $i=g+s+j$,}\\
  v_{n,g}+e_{s+n-1,1}+a_{1,1}-u_{1,1}\\
  \qquad+2a_{2,2}+\ldots+ 2a_{g-1,g-1}&\text{if $j=n$ and $g$ is odd,}\\
   v_{n,g}+e_{s+n-1,1}+a_{1,1}\\
   \qquad 2a_{3,3}+\ldots+2a_{g-1,g-1}  &\text{if $j=n$ and $g$ is even.}
  \end{cases}
 \end{equation*} 
\end{prop}
\begin{prop}\label{prop:kernel:3}
 Let $g\geq 3$, $n\geq 2$ and $G={\Cal{M}}(N_{g,s}^n)$. Then $\gen{\kre{X}}\cap \ker\kre{\partial}_1$ is the abelian group generated by generators specified in the statement of Proposition \ref{prop:kernel:1}, Generators (K12) with $j=n$ specified in the statement of Proposition \ref{prop:kernel:2}, and additionally
 \begin{enumerate}
  \item[(K13)] $s_{j,i}$ if $j\leq n-1$ and $i\not\in\{g+s+j,g+s+j+1\},$
  \item[(K14)] $s_{j,g+s+j}+s_{j,g+s+j+1}$ if $j<n-1$,
  \item[(K15)] $s_{j,g+s+j}-e_{s+j-1,1}+ 2e_{s+j,1}-e_{s+j+1,1}$ if $j<n-1$,
  \item[(K16)] 
  \[\begin{cases}
  s_{n-1,g+s+n-1}+2e_{s+n-1,1}-e_{s+n-2,1}+a_{1,1}\\
  \qquad-u_{1,1}+2a_{2,2}+\ldots+2a_{g-1,g-1}&\text{if $g$ is odd,}\\
   s_{n-1,g+s+n-1}+2e_{s+n-1,1}-e_{s+n-2,1}+a_{1,1}\\
   \qquad+2a_{3,3}+\ldots+2a_{g-1,g-1}  &\text{if $g$ is even.}
    \end{cases}
\]
 \end{enumerate}
\end{prop}
 \section{Bounding $H_1({\Cal{PM}}^+(N_{g,s}^n);H_1(N_{g,s}^n;\zz))$ from above}\label{sec:above}
In this section we will use the formula \eqref{eq_rew_rel} to rewrite some relations between generators specified in Theorem \ref{tw:gen:pure:plus} as relations between homology classes. Our goal is to reduce these generating sets for homology groups to the ones specified in Propositions \ref{h1:g3} and \ref{h1:g4567} below.

Let 
\[\map{i}{N_{g,1}}{N_{g,s}^n}\]
be an embedding of a non--orientable subsurface of genus $g$ with one boundary component such that $N_{g,1}$ is disjoint from $\delta_1,\ldots,\delta_{s+n}$ (the complement of $N_{g,1}$ in $N_{g,s}^n$ is a disk containing $\delta_1,\ldots,\delta_{s+n}$). This embedding induces homomorphisms
\[\begin{CD}
   {\Cal{M}}(N_{g,1})@>i_*>> {\Cal{PM}}^+(N_{g,s}^n)\\
   @VV\ro V @VV\ro V\\
   \textrm{Aut}(H_1(N_{g,1};\zz)) @>i_*>> \textrm{Aut}(H_1(N_{g,s}^n;\zz))
  \end{CD}
\]
This leads to the following homomorphism 
\[\begin{CD}H_1({\Cal{M}}(N_{g,1});H_1(N_{g,1};\zz)) @>i_* >> H_1({\Cal{PM}}^+(N_{g,s}^n);H_1(N_{g,s}^n;\zz)).\end{CD}\]
Moreover, some of the generators specified in the statement of Proposition~\ref{prop:kernel:1} are images under this homomorphism of generators used in Proposition~4.2 of \cite{Stukow_homolTopApp} to compute $H_1({\Cal{M}}(N_{g,1});H_1(N_{g,1};\zz))$. This allows to transfer (via $i_*$) some of the relations between these generators obtained in Section~5 of \cite{Stukow_homolTopApp}. In particular,
\begin{itemize}
 \item Generators (K1): $a_{j,i}$ for $j=1,\ldots,g-1$,  $i=1,\ldots j-1,j+2,\ldots,g$ generate a cyclic group of order at most 2. They are trivial if $g\geq 7$.
 \item Generators (K2) generate a cyclic group of order at most 2. They are trivial if $g\geq 4$.
 \item Generators (K3): $u_{1,i}$ for $i=1,\ldots, g-2$ generate a cyclic group of order at most 2.
 \item Generator (K4) is trivial.
 \item Generators (K8): $b_{1,i}$ for $i=5,\ldots,g$ are superfluous (they can be expressed in terms of generators (K1)).
 \item Generators (K9), (K10) are trivial.
 \item Generator (K11) has order at most 2.
\end{itemize}
The formula \eqref{eq_rew_rel} and the relation
\[a_ja_{j+1}a_{j}=a_{j+1}a_ja_{j+1},\quad\text{ for $j=1,\ldots,g-2$},\]
imply that for $i>g$
\[ \begin{aligned}
   0&=([a_{j}]+a_{j}[a_{j+1}]+a_{j}a_{j+1}[a_{j}]-[a_{j+1}]-a_{j+1}[a_{j}]-a_{j+1}a_{j}[a_{j+1}])\otimes \xi_i\\
   &=a_{j,i}+a_{j+1,i}+a_{j,i}-a_{j+1,i}-a_{j,i}-a_{j+1,i}=a_{j,i}-a_{j+1,i}.
  \end{aligned}\]
  Hence 
  \begin{equation}\label{eq:ajg}
   a_{j,i}=a_{1,i}\quad\text{ for $j=1,\ldots,g-1$, $i>g$.}
  \end{equation}

If $s+n\geq 2$, then the relation
\[a_1e_j=e_ja_1\quad\text{ for $j=1,\ldots,s+n-1$}\]
gives
\[\begin{aligned}
   0&=([a_1]+a_1[e_j]-[e_j]-e_j[a_1])\otimes \xi_i\\
   &=[a_1]\otimes (I-\psi(e_j^{-1}))\xi_i-[e_j]\otimes(I-\psi(a_1^{-1}))\xi_i\\
   &=\pm\begin{cases}
     (a_{1,1}+a_{1,2})+a_{1,g+1}+\ldots+a_{1,g+j}-(e_{j,1}+e_{j,2})&\text{if $i=1,2$,}\\
     0&\text{if $i>2$.}
     \end{cases}
  \end{aligned}
\]
This relation implies that Generators (K6) are superfluous
\begin{equation}\label{eq:e12}
   e_{j,1}+e_{j,2}=(a_{1,1}+a_{1,2})+a_{1,g+1}+\ldots+a_{1,g+j}.
  \end{equation}
%
The braid relation
\[a_2e_ja_2=e_ja_2e_j\quad\text{ for $j=1,\ldots,s+n-1$}\]
gives
\[\begin{aligned}
   0&=([a_{2}]+a_{2}[e_{j}]+a_{2}e_{j}[a_{2}]-[e_{j}]-e_{j}[a_{2}]-e_{j}a_{2}[e_{j}])\otimes \xi_i\\
   &=[a_2]\otimes (I+\psi(e_{j}^{-1}a_2^{-1})-\psi(e_{j}^{-1}))\gamma_i\\
&+[e_{j}]\otimes(\psi(a_2^{-1})-I-\psi(a_2^{-1}e_{j}^{-1}))\gamma_i\\
&=\begin{cases}
                 a_{2,i}-e_{j,i}&\text{\hspace{-0cm}if $i\not\in\{1,2,3\}$,}\\
 a_{2,1}-e_{j,3}+a_{2,g+1}+\ldots+a_{2,g+j}&\text{\hspace{0cm}if $i=3$,}\\
(*)+(a_{2,2}+a_{2,3})+(e_{j,1}+e_{j,2})+e_{j,g+1}+\ldots+e_{j,g+j}&\text{\hspace{0cm}if $i=2$,}\\
(*)&\text{if $i=1$.}
                 \end{cases}
  \end{aligned}\]
In the above formula $(*)$ denotes some expression homologous to 0 by previously obtained relations. As we progress further, we will often perform simplifications based on previously obtained relations, from now on we will use symbol '$\equiv$' in such cases.
  
The first two cases of this relation and the formula \eqref{eq:ajg} imply that Generators (K5)
\begin{equation}\label{eq:eji}
 e_{j,i}=\begin{cases}
          a_{1,i}&\text{if $i\geq 4$,}\\
          a_{2,1}+a_{1,g+1}+\ldots+a_{1,g+j}&\text{if $i=3$.}
         \end{cases}
\end{equation}
are superfluous. 


The third case together with formulas \eqref{eq:e12} and \eqref{eq:eji} imply that
\begin{equation*}
 2a_{1,g+j}=0\quad\text{ for $j=1,2,\ldots,s+n-1$,}
\end{equation*}
or equivalently
\begin{equation*}
 2(e_{j,1}+e_{j,2})=0.
\end{equation*}

The relation
\[a_j d_k=d_ka_j\quad\text{ for $j=1,\ldots,g-1$, $k=1,\ldots,s-1$}\]
gives
\[\begin{aligned}
   0&=([a_j]+a_j[d_k]-[d_k]-d_k[a_j])\otimes \xi_i\\
   &=[a_j]\otimes (I-\psi(d_k^{-1}))\xi_i-[d_k]\otimes(I-\psi(a_j^{-1}))\xi_i\\
   &=\pm\begin{cases}
     d_{k,j}+d_{k,j+1}&\text{if $i=j,j+1$,}\\
     0&\text{if $i\not\in\{j,j+1\}$.}
     \end{cases}
  \end{aligned}
\]
This implies that Generators (K7)
\begin{equation}\label{eq:u1}
d_{k,j}=(-1)^{j-1}d_{k,1}\quad \text{for $j=2,\ldots,g$, $k=1,\ldots,s-1$}
\end{equation}
are superfluous.

Similarly, the relation
\[u_{1}d_k=d_ku_{1} \]
implies that 
\[0=[d_k]\otimes(\psi(u_{1}^{-1})-I)\gamma_{2}=d_{k,1}-d_{k,2}=0,\quad\text{ for $k=1,\ldots,s-1$,}\]
which together with the formula \eqref{eq:u1} implies that
\begin{equation*}
 2d_{k,1}=0\quad\text{for $k=1,\ldots,s-1$}.
\end{equation*}
Relation 
\[e_jd_k=d_ke_j\]
implies that
\[\begin{aligned}
0=[d_k]\otimes(\psi(e_{j}^{-1})-I)\gamma_1&\equiv d_{k,g+1}+\ldots+d_{k,g+j}\\
&\text{ for $k=1,\ldots,s$, $j=1,\ldots,s+n-1$.}
 \end{aligned}\]
 This implies that Generators (K7): $d_{j,i}$ are trivial for $i>g$.

 Suppose now that $g=3$ and consider the relation
 \begin{equation}\label{rel:u2}
  (u_1e_{s+n})^2=\widehat{d}_{s+n}=(a_1a_2)^6,
 \end{equation}
 where $\widehat{d}_{s+n}$ is defined as in Section \ref{sec:gen}. The right--hand side of this relation is a chain relation, and the left--hand side is a square of a crosscap slide (see \cite{Szep_curv}, Theorem 7.17, Relation (8)). If $i>3$, and
 \[\begin{aligned}
   M&=I+\psi(a_2^{-1}a_1^{-1})+\psi(a_2^{-1}a_1^{-1})^2+\ldots+\psi(a_2^{-1}a_1^{-1})^5,\\
   N&=I+\psi(e_{s+n}^{-1}u_1^{-1}), 
   \end{aligned}
\]
 then Relation \eqref{rel:u2} gives
 \[0=[u_1]\otimes N\xi_i+[e_{s+n}]\otimes \psi({u_1}^{-1})N\xi_i-[a_1]\otimes M\xi_i-[a_2]\otimes\psi(a_1^{-1})M\xi_i.\]
 If we now assume that $i>g$, then we get
 \begin{equation}
 \begin{aligned}\label{g3:ui2}
  0&=2[u_1]\otimes \xi_i+2[e_{s+n}]\otimes \xi_i-6[a_1]\otimes \xi_i-6[a_2]\otimes\xi_i\\
  &=2u_{1,i}+2[e_{s+n}]\otimes \xi_i.
 \end{aligned} 
 \end{equation}
 As we observed in the formula \eqref{di:3:esn},
 \[e_{s+n}a_2u_2a_1=a_2u_2,\]
 hence if $i>3$, then
\begin{equation}
 \begin{aligned}\label{g3:ui2:es}
0&=([e_{s+n}]+[a_2]+[u_2]+[a_1]-[a_2]-[u_2])\otimes \xi_i\\
&=[e_{s+n}]\otimes\xi_i+a_{1,i}.
   \end{aligned}
\end{equation}
By combining formulas \eqref{g3:ui2} and \eqref{g3:ui2:es} we get
\[2u_{1,i}=0,\quad \text{for $i>3$}.\]
Note that at this point we proved 
\begin{prop}\label{h1:g3}
Let $s+n\geq 1$. Then $H_1({\Cal{PM}}^+(N_{3,s}^n);H_1(N_{3,s}^n;\zz))$ is generated by 
\[\begin{aligned}
\{&a_{1,1}+a_{1,2},a_{1,3},a_{1,4},\ldots,a_{1,2+s+n},\\
&u_{1,3},u_{1,4},\ldots,u_{1,2+s+n},d_{1,1},\ldots,d_{s-1,1}\}
\end{aligned} \]
and each of these generators has order at most 2. \qed
\end{prop}
For the rest of this section assume that $g\geq 4$.

The relation
\[e_ja_3=a_3e_j\quad \text{for $j=1,\ldots,s+n-1$,}\]
gives
\[\begin{aligned}
0&=[e_j]\otimes (I-\psi(a_3^{-1}))\gamma_1-[a_3]\otimes(I-\psi(e_j^{-1}))\gamma_1=\\
&=a_{3,1}+a_{3,2}+a_{3,g+1}+\ldots+a_{3,g+j}.
\end{aligned}
\]
Together with the formula \eqref{eq:ajg} this implies that Generators (K1): $a_{j,i}$ are trivial for $i>g$.

Observe that relations
\[u_ju_{j+1}u_j=u_{j+1}u_ju_{j+1},\quad \text{for $j=1,2$,}\]
easily imply that 
\[u_{3,i}=u_{2,i}=u_{1,i},\quad\text{for $i>g$}.\]

Hence, the relation
\[e_ju_{3}=u_{3}e_j,\quad \text{for $j=1,\ldots,s+n-1$,}\]
gives
\[\begin{aligned}
0&=[e_j]\otimes (I-\psi(u_{3}^{-1}))\gamma_1-[u_{3}]\otimes(I-\psi(e_j^{-1}))\gamma_1=\\
&=u_{3,1}+u_{3,2}+u_{3,g+1}+\ldots+u_{3,g+j}=\\
&=u_{3,1}+u_{3,2}+u_{1,g+1}+\ldots+u_{1,g+j}.
\end{aligned}
\]
This implies that Generators (K3): $u_{1,i}$ are trivial for $i>g$.

Relation 
\[e_jb_1=b_1e_j,\quad \text{for $j=1,\ldots,s+n-1$,}\]
gives
\[\begin{aligned}
0&=[e_j]\otimes (I-\psi(b_{1}^{-1}))(\gamma_1-\gamma_3)-[b_{1}]\otimes(I-\psi(e_j^{-1}))(\gamma_1-\gamma_3)=\\
&=b_{1,1}+b_{1,2}+b_{1,g+1}+\ldots+b_{1,g+j}.
\end{aligned}
\]
This implies that Generators (K8): $b_{1,i}$ are trivial for $i>g$.

At this point we proved 
\begin{prop}\label{h1:g4567}
 Let $g>3$ and $s+n\geq 1$. Then $H_1({\Cal{PM}}^+(N_{g,s}^n);H_1(N_{g,s}^n;\zz))$ is generated by
\[
\begin{aligned}
&\begin{cases}
    \left\{a_{1,3},u_{1,3},b_{1,1}-a_{1,1}-a_{3,3}, d_{1,1},\ldots,d_{s-1,1}\right\}&\text{if $g=4$,}\\
            \left\{a_{1,3},u_{1,3},b_{1,1}-a_{1,1}-a_{3,3}\right\}&\text{if $g=5,6$,}\\
            \left\{u_{1,3},b_{1,1}-a_{1,1}-a_{3,3}\right\}&\text{if $g\geq 7$,}
                                                 \end{cases}
                                                 \end{aligned} 
\]
and each of these generators has order at most 2. \qed
\end{prop}
\section{Bounding $H_1({\Cal{PM}}^{k}(N_{g,s}^n);H_1(N_{g,s}^n;\zz))$ from above}\label{secK:above}
As in the previous section, we will use the formula \eqref{eq_rew_rel} to reduce the generating set for the group  $H_1({\Cal{PM}}^{+}(N_{g,s}^n);H_1(N_{g,s}^n;\zz))$ to the one specified in the statements of Propositions \ref{h1v:g4567} and \ref{h1v:g3} below.

By Proposition \ref{prop:kernel:2}, $H_1({\Cal{PM}}^{k}(N_{g,s}^n);H_1(N_{g,s}^n;\zz))$ is generated by generators of the group $H_1({\Cal{PM}}^{+}(N_{g,s}^n);H_1(N_{g,s}^n;\zz))$ and additionally Generators (K12) corresponding to puncture slides $v_{k+1},\ldots,v_{n}$ (see Proposition \ref{tw:gen:pure}). All the computations from the previous section hold true, hence $H_1({\Cal{PM}}^{k}(N_{g,s}^n);H_1(N_{g,s}^n;\zz))$ is generated by Generators (K12): $\widetilde{v}_{j,i}$ and elements  specified in the statements of Propositions \ref{h1:g3} and \ref{h1:g4567}.

Note that for any $x\in {\Cal{PM}}^{+}(N_{g,s}^n)$ and $k<j\leq n$,
\[y=v_{j}^{-1}xv_{j}\in{\Cal{PM}}^{+}(N_{g,s}^n),\]
hence both $x$ and $y$ are products of generators of ${\Cal{PM}}^{+}(N_{g,s}^n)$ (that is these products do not contain puncture slides). Therefore, the relation
\[xv_j=v_jy\]
gives
\[
 \begin{aligned}
0&=[x]\otimes\gamma_i+[v_j]\otimes \psi (x^{-1})\gamma_i-[v_j]\otimes \gamma_i-[y]\otimes \psi(v_j^{-1})\gamma_i=\\
&=[v_j]\otimes (\psi (x^{-1})-I)\gamma_i+A_{j,i}^x=
\sum_{r=1}^{g+s+n-1} m_rv_{j,r}+A_{j,i}^x=\\
&=\sum_{r=1}^{g+s+n-1} m_r\widetilde{v}_{j,r}+\widetilde{A}_{j,i}^x,
 \end{aligned}
\]
for some coefficients $m_r$ and expressions $A_{j,i}^x$, $\widetilde{A}_{j,i}^x$ which contain neither $v_{j,r}$ nor $\widetilde{v}_{j,r}$. Moreover, by Proposition \ref{prop:kernel:2}, 
\[\sum_{r=1}^{g+s+n-1} m_r\widetilde{v}_{j,r}=-\widetilde{A}_{j,i}^x\]
is an element of the kernel $\ker\kre{\partial}_1$, hence this element is a linear combination of generators specified in the statements of Propositions \ref{h1:g3} and \ref{h1:g4567}.

Now we use the above general analysis to two special cases: $x=a_i$, $i=1,\ldots,g-1$ and $x=e_i$, $i=1,\ldots,s+n-1$. 

In the first case we get
\[0=[v_j]\otimes (\psi (a_i^{-1})-I)\gamma_i+A_{j,i}^{a_i}=
\widetilde{v}_{j,i}+\widetilde{v}_{j,i+1}+\widetilde{A}_{j,i}^{a_i}.\]
This implies that generators $\widetilde{v}_{j,2},\ldots,\widetilde{v}_{j,g}$ are superfluous.

In the second case we get
\[0=[v_j]\otimes (\psi (e_i^{-1})-I)\gamma_i+A_{j,i}^{e_i}=
\widetilde{v}_{j,1}+\widetilde{v}_{j,2}+\widetilde{v}_{j,g+1}+\ldots+\widetilde{v}_{j,g+i}+\widetilde{A}_{j,i}^{a_i}.\]
This implies that generators $\widetilde{v}_{j,g+1},\ldots,\widetilde{v}_{j,g+s+n-1}$ are superfluous. 

Relations $a_1v_j=v_ja_1$ and $u_1v_j=v_ju_1$ give 
\[\begin{aligned}
   0&=[v_j]\otimes (\psi (a_i^{-1})-I)\gamma_i-[a_1]\otimes (\psi (v_j^{-1})-I)\gamma_i=v_{j,1}+v_{j,2},\\
   0&=[v_j]\otimes (\psi (u_i^{-1})-I)\gamma_i-[u_1]\otimes (\psi (v_j^{-1})-I)\gamma_i=v_{j,1}-v_{j,2}.
  \end{aligned}
\]
respectively. This implies that $2v_{j,1}=2\widetilde{v}_{j,1}=0$ and we proved

\begin{prop}\label{h1v:g4567}
 Let $g> 3$, $s+n\geq 1$ and $0\leq k\leq n$. Then the group $H_1({\Cal{PM}}^{k}(N_{g,s}^n);H_1(N_{g,s}^n;\zz))$ is generated by
\[
\begin{aligned}
&\begin{cases}
    \left\{a_{1,3},u_{1,3},b_{1,1}-a_{1,1}-a_{3,3}, d_{1,1},\ldots,d_{s-1,1},v_{k+1,1},\ldots,v_{n,1}\right\}&\text{if $g=4$,}\\
            \left\{a_{1,3},u_{1,3},b_{1,1}-a_{1,1}-a_{3,3},v_{k+1,1},\ldots,v_{n,1}\right\}&\text{if $g=5,6$,}\\
            \left\{u_{1,3},b_{1,1}-a_{1,1}-a_{3,3},v_{k+1,1},\ldots,v_{n,1}\right\}&\text{if $g\geq 7$,}
                                                 \end{cases}
                                                 \end{aligned} 
\]
and each of these generators has order at most 2. \qed
\end{prop}
For the rest of this section assume that $g=3$.

If $j<n$, then relations $a_1v_j=v_ja_1$ and $u_1v_j=v_ju_1$ give
\[\begin{aligned}
0&=[v_j]\otimes (\psi (a_1^{-1})-I)\gamma_3-[a_1]\otimes (\psi (v_j^{-1})-I)\gamma_3=-a_{1,3+s+j},\\
0&=[v_j]\otimes (\psi (u_1^{-1})-I)\gamma_3-[u_1]\otimes (\psi (v_j^{-1})-I)\gamma_3=-u_{1,3+s+j},
  \end{aligned}
\]
respectively. Hence, $a_{1,i}=u_{1,i}=0$ if $i>3+s+k$.

Finally, if $s+k>0$, then relations $a_1v_n=v_na_1$ and $u_1v_n=v_nu_1$ give
\[\begin{aligned}
   0&=[v_n]\otimes (\psi (a_1^{-1})-I)\gamma_3-[a_1]\otimes (\psi (v_n^{-1})-I)\gamma_3=\\
   &=2(a_{1,1}+a_{1,2})+2a_{1,3}+a_{1,4}+\ldots+a_{1,2+s+n}=\\
   &=a_{1,4}+\ldots+a_{1,3+s+k},\\
   0&=[v_n]\otimes (\psi (u_1^{-1})-I)\gamma_3-[u_1]\otimes (\psi (v_n^{-1})-I)\gamma_3=\\
   &=2(u_{1,1}+u_{1,2})+2u_{1,3}+u_{1,4}+\ldots+u_{1,2+s+n}=\\
   &=u_{1,4}+\ldots+u_{1,3+s+k},
  \end{aligned}
\]
respectively. Hence, $a_{1,3+s+k}$ and $u_{1,3+s+k}$ are superfluous provided $s+k>0$. This proves 
\begin{prop}\label{h1v:g3}
Let $s+n\geq 1$ and $0\leq k\leq n$. Then $H_1({\Cal{PM}}^{k}(N_{3,s}^n);H_1(N_{3,s}^n;\zz))$ is generated by 
\[
\begin{cases}
\left\{a_{1,1}+a_{1,2},a_{1,3},u_{1,3},v_{k+1,1},\ldots,v_{n,1}\right\}&\text{ if $s+k=0$},  \\
\{a_{1,1}+a_{1,2},a_{1,3},a_{1,4},\ldots,a_{1,2+s+k},\\
\quad u_{1,3},u_{1,4},\ldots,u_{1,2+s+k},\\
\quad d_{1,1},\ldots,d_{s-1,1}, v_{k+1,1},\ldots,v_{n,1}\}&\text{ if $s+k>0$,}
\end{cases}
\]
and each of these generators has order at most 2. \qed
\end{prop}
\section{Bounding $H_1({\Cal{M}}(N_{g,s}^n);H_1(N_{g,s}^n;\zz))$ from above}\label{secM:above}
As in the previous two sections, we will use the formula \eqref{eq_rew_rel} to reduce the generating set for the group  $H_1({\Cal{M}}(N_{g,s}^n);H_1(N_{g,s}^n;\zz))$ to the one specified in the statement of Proposition \ref{prop:h1:vs} below.

By Proposition \ref{prop:kernel:3}, $H_1({\Cal{M}}(N_{g,s}^n);H_1(N_{g,s}^n;\zz))$ is generated by generators of the group $H_1({\Cal{PM}}^{+}(N_{g,s}^n);H_1(N_{g,s}^n;\zz))$, Generators (K12) corresponding to puncture slide $v_n$ and additionally Generators (K13)--(K16) corresponding to elementary braids: $s_1,\ldots,s_{n-1}$. All computations from the previous two sections hold true, hence $H_1({\Cal{M}}(N_{g,s}^n);H_1(N_{g,s}^n;\zz))$ is generated by generator (K12): $v_{n,1}$, generators (K13)--(K16), and elements  specified in the statements of Propositions \ref{h1:g3} and \ref{h1:g4567}. Moreover if $g=3$ and $i\geq 3+s$, then $a_{1,i}$ and $u_{1,i}$ are superfluous.

If $i\neq s+j$, then the relation
\[e_is_j=s_je_i\]
gives
\[\begin{aligned}
  0&=[s_j]\otimes (\psi (e_i^{-1})-I)\gamma_1-[e_i]\otimes (\psi (s_j^{-1})-I)\gamma_1 \\
   &=s_{j,1}+s_{j,2}+s_{j,g+1}+\ldots+s_{j,g+i}=A_{j,i}.
  \end{aligned}
\]
In particular, Generator (K14)
\[s_{j,g+s+j}+s_{j,g+s+j+1}=A_{j,s+j+1}-A_{j,s+j-1}=0\]
is trivial and Generators (K13) of the form
\[s_{j,g+i}=\begin{cases}
             A_{j,i}-A_{j,i-1}=0&\text{if $i>1$,}\\
             A_{j,1}-(s_{j,1}+s_{j,2})&\text{if $i=1$}
            \end{cases}\]
are superfluous.

The relation
\[s_ja_i=a_is_j\quad \text{ for $i<g$}\]
gives
\begin{equation}\label{s:cyc:A}
   0=[s_j]\otimes (\psi (a_i^{-1})-I)\gamma_i-[a_i]\otimes (\psi (s_j^{-1})-I)\gamma_i=s_{j,i}+s_{j,i+1}.
\end{equation}
Relation $s_{j}s_{j+1}s_j=s_{j+1}s_js_{j+1}$ gives
\begin{equation*}
\begin{aligned}
   0&=([s_{j}]+s_{j}[s_{j+1}]+s_{j}s_{j+1}[s_{j}]-[s_{j+1}]-s_{j+1}[s_{j}]-s_{j+1}s_{j}[s_{j+1}])\otimes \gamma_1\\
   &=s_{j,i}+s_{j+1,i}+s_{j,i}-s_{j+1,i}-s_{j,i}-s_{j+1,i}=s_{j,i}-s_{j+1,i}.
\end{aligned}
\end{equation*}
This together with the formula \eqref{s:cyc:A} implies that Generators (K13) generate a cyclic group. Moreover, the relation
\[s_ju_1=u_1s_j\]
implies that
\[ 0=[s_j]\otimes (\psi (u_1^{-1})-I)\gamma_1-[u_1]\otimes (\psi (s_j^{-1})-I)\gamma_1=s_{j,1}-s_{j,2},\]
which together with the formula \eqref{s:cyc:A} implies that the cyclic group generated by generators (K13) has order at most two.

By Proposition \ref{prop:rel:sj:ej}, 
\[e_{s+j-1}e_{s+j+1}s_j=e_{s+j}s_j^3e_{s+j},\]
and this relation gives 
\begin{equation}\label{es:rel:hom}
\begin{aligned}0=&([e_{s+j-1}]+e_{s+j-1}[e_{s+j+1}]+e_{s+j-1}e_{s+j+1}[s_j])\otimes\gamma_1\\
   &-([e_{s+j}]+e_{s+j}(1+s_j+s_j^2)[s_j]-e_{s+j}s_j^3[e_{s+j}])\otimes\gamma_1\\
   =&e_{s+j-1,1}+[e_{s+j}]\otimes(-I-\psi(s_j^{-3})\psi(e_{s+j}^{-1})\gamma_1+[e_{s+j+1}]\otimes\psi(e_{s+j-1}^{-1})\gamma_1\\
   &+[s_j]\otimes \left(\psi(e_{s+j+1}^{-1})\psi(e_{s+j-1}^{-1})-(I+\psi(s_j^{-1})+\psi(s_j^{-2}))\psi(e_{s+j}^{-1})\right)\gamma_1.
  \end{aligned} 
\end{equation}
If $j<n-1$ this gives
\[\begin{aligned}
   0&=(s_{j,g+s+j}-e_{s+j-1,1}+2e_{s+j,1}-e_{s+j+1,1})\\
   &+(e_{s+j,1}+e_{s+j,2})+(e_{s+j,g+1}+\ldots+e_{s+j,g+s+j-1}+e_{s+j,g+s+j+1})\\
   &-(e_{s+j+1,1}+e_{s+j+1,2})-(e_{s+j+1,g+1}+\ldots+e_{s+j+1,g+s+j-1})+(*).
  \end{aligned}
\]
This implies that Generator (K15) is superfluous. If $j=n-1$, then the formula \eqref{es:rel:hom} yields a more complicated expression, however it is also of the form
\[0=s_{n-1,g+s+n-1}+(*),\]
where $(*)$ denotes some expression which does not contain $s_{j,i}$. This implies that Generator (K16) is also superfluous, and we proved that
\begin{prop}\label{prop:h1:vs}
 Let $g\geq 3$ and $n\geq 2$. Then $H_1({\Cal{M}}(N_{g,s}^n);H_1(N_{g,s}^n;\zz))$ is generated by 
\[
\begin{cases}
a_{1,1}+a_{1,2},a_{1,3},u_{1,3},v_{n,1},s_{1,1}&\text{ if $g=3$ and $s=0$},  \\
a_{1,1}+a_{1,2},a_{1,3},a_{1,4},\ldots,a_{1,2+s},\\
\quad u_{1,3},u_{1,4},\ldots,u_{1,2+s},\\
\quad d_{1,1},\ldots,d_{s-1,1}, v_{n,1},s_{1,1}&\text{ if $g=3$ and $s>0$},\\
a_{1,3},u_{1,3},b_{1,1}-a_{1,1}-a_{3,3},d_{1,1},\ldots,d_{s-1,1},v_{n,1},s_{1,1}&\text{ if $g=4$},  \\
a_{1,3},u_{1,3},b_{1,1}-a_{1,1}-a_{3,3},v_{n,1},s_{1,1}&\text{ if $g=5$ or $g=6$},  \\
u_{1,3},b_{1,1}-a_{1,1}-a_{3,3},v_{n,1},s_{1,1}&\text{ if $g\geq 7$},
\end{cases}
\]
and each of these generators has order at most 2. \qed
\end{prop}
\section{Bounding $H_1({\Cal{PM}}^k(N_{g,s}^n);H_1(N_{g,s}^n;\zz))$ from below}   \label{sec:pmk:below}                                              
In this section we use various quotients of ${\Cal{PM}}^k(N_{g,s}^n)$ in order to prove that all homology classes specified in Propositions \ref{h1v:g4567} and \ref{h1v:g3} are nontrivial. This will complete the proof of Theorem \ref{MainThm1}.

If we glue a disk to each boundary components of $N_{g,s}^n$ and forget about punctures, then we get a closed non--orientable surface $N_g$ of genus $g$. If
\[\map{i}{N_{g,s}^n}{N_{g}}\]
is the corresponding inclusion map, then $i$ induces homomorphisms
\[\begin{CD}
   {\Cal{PM}}^k(N_{g,s}^n)@>i_*>> {\Cal{M}}(N_{g})\\
   @VV\psi V @VV\psi V\\
   \textrm{Aut}(H_1(N_{g,s}^n;\zz)) @>i_*>> \textrm{Aut}(H_1(N_{g};\zz))
  \end{CD}
\]
This leads to the following homomorphism 
\[\begin{CD}H_1({\Cal{PM}}^k(N_{g,s}^n);H_1(N_{g,s}^n;\zz)) @>i_* >> H_1({\Cal{M}}(N_{g});H_1(N_{g};\zz)).\end{CD}\]
Moreover, by Theorem 1.1 of \cite{Stukow_homolTopApp} (see the very last formula in the proof of that theorem), we have
\begin{equation*}
 \begin{aligned}
 i_*(a_{1,1}+a_{1,2})&\neq 0 \quad\text{if $g=3$,}\\
  i_*(u_{1,3})&\neq 0 \quad\text{if $g\geq 3$,}\\
   i_*(a_{1,3})&\neq 0 \quad\text{if $g\in\{3,4,5,6\}$,}\\
   i_*(b_{1,1}-a_{1,1}-a_{3,3})&\neq 0 \quad\text{if $g\geq 4$}
  \end{aligned}
\end{equation*}
and all these classes are linearly independent.

In order to prove that homology classes corresponding to puncture slides are nontrivial, fix $k<j\leq n$ and consider the following homomorphisms
\[\begin{aligned}
 &\map{\alpha}{{\Cal{PM}}^+(N_{g,s}^n)}{\zz_2},\\
 &\map{\beta}{H_1(N_{g,s}^n;\zz)}{\zz_2}.
\end{aligned}\]
The first homomorphism is defined as follows: $\alpha(f)=1$ if and only if $f$ changes the local orientation around the puncture $z_j$. The second one is the composition 
\[\begin{CD}
H_1(N_{g,s}^n;\zz)@>>>H_1(N_{g,s}^n;\zz_2)@>>>\gen{\gamma_1}.
\end{CD}\]
of the reduction to $\zz_2$ coefficients and the projection: 
\[\begin{cases}
  \gamma_1,\gamma_2,\ldots,\gamma_g \longmapsto \gamma_1,\\
  \delta_1,\delta_2,\ldots,\delta_{s+n-1}\longmapsto 0.
  \end{cases}
\]
It is straightforward to check that for any $m\in H_1(N_{g,s}^n;\zz)$ and $f\in {\Cal{PM}}^+(N_{g,s}^n)$
\[\beta(f(m))=\beta(m).\]
Hence, if we regard $\gen{\gamma}$ as a trivial $\alpha({\Cal{PM}}^+(N_{g,s}^n))$ module, then $(\alpha,\beta)$ induce homomorphism
\[\begin{CD}H_1({\Cal{PM}}^k(N_{g,s}^n);H_1(N_{g,s}^n;\zz)) @>(\alpha,\beta) >> H_1(\zz_2;\zz_2)\cong \zz_2. \end{CD}\]
Moreover, if $x$ is one of the generators specified in the statements of Propositions \ref{h1v:g4567} and \ref{h1v:g3}, then
\[(\alpha,\beta)(x)\neq 0\iff x=v_{j,1}.\]
This implies that $v_{j,1}$ is nontrivial and independent from other generators.

If $s\geq 2$ and $g\leq 4$, then for any fixed $1\leq j\leq s-1$ there is a homomorphism 
\[\begin{CD}H_1({\Cal{PM}}^k(N_{g,s}^n);H_1(N_{g,s}^n;\zz)) @>i_* >> H_1({\Cal{M}}(N_{g+2});H_1(N_{g+2};\zz))\end{CD}\]
induced by the inclusion $\map{i}{N_{g,s}^n}{N_{g+2}}$, where $N_{g+2}$ is a closed non--orientable surface of genus $g+2$
obtained from $N_{g,s}^n$, by forgetting the punctures, connecting boundary components of numbers $1\leq j\leq s-1$ and $s$ by a cylinder, and gluing a disk to all the remaining boundary components.

Moreover, $\delta_j$ becomes a two--sided nonseparating circle in $N_{g+2}$, hence we can choose generators for ${\Cal{M}}(N_{g+2})$ so that 
\[i_*(d_{j,1})=a_{1,3}\in H_1({\Cal{M}}(N_{g+2});H_1(N_{g+2};\zz)).\]
By Theorem 1.1 of \cite{Stukow_homolTopApp}, this homology class is nontrivial provided $g+2\leq 6$.
This completes the proof of Theorem~\ref{MainThm1} if $g>3$ or $s+k\leq 1$. 

Hence, assume that $g=3$ and $s+k\geq 2$. Fix $1\leq j\leq s+k-1$ and and glue a disk with a puncture to each boundary component of $N_{3,s}^n$. Then forget about all the punctures except those with numbers $j$ and $s+k$. As a result we obtain an inclusion 
\[\map{i}{N_{3,s}^n}{N_{3}^2},\]
which leads to a homomorphism
\[\begin{CD}H_1({\Cal{PM}}^k(N_{g,s}^n);H_1(N_{g,s}^n;\zz)) @>i_* >> H_1({\Cal{M}}(N_{3}^2);H_1(N_{3}^2;\zz)).\end{CD}\]
By Theorem 1.2 of \cite{PawlakStukowHomoPuncturedG3},
\[\begin{aligned}
   i_*(a_{1,3+j})=a_{1,4}\neq 0,\\
   i_*(u_{1,3+j})=u_{1,4}\neq 0.
  \end{aligned}
\]
This implies that generators $a_{1,4},\ldots,a_{1,2+s+k}$, $u_{1,4},\ldots,u_{1,2+s+k}$ are nontrivial and linearly independent. This concludes the proof of Theorem \ref{MainThm1}.
\section{Bounding $H_1({\Cal{M}}(N_{g,s}^n);H_1(N_{g,s}^n;\zz))$ from below}  
\label{sec:pm:below}                                              
In this section we will show that all generators specified in the statement of Proposition \ref{prop:h1:vs} are nontrivial and linearly independent. This will complete the proof of Theorem \ref{MainThm2}.

As in the previous section, we argue that homology classes
\begin{equation}\label{eq:prev:gen}  \begin{aligned}
\begin{cases}
 a_{1,1}+a_{1,2}&\text{if $g=3$,}\\
 a_{1,4},\ldots a_{1,2+s}&\text{if $g=3$ and $s>1$,}\\
 u_{1,4},\ldots u_{1,2+s}&\text{if $g=3$ and $s>1$,}\\
 a_{1,3}&\text{if $g<7$,}\\
 u_{1,3},\\
 b_{1,1}-a_{1,1}-a_{3,3}&\text{ if $g\geq 4$,}\\
 d_{j,1}&\text{if $g\leq 4$ and $1\leq j\leq s-1$}
\end{cases}
  \end{aligned}
\end{equation}
are nontrivial and independent. Hence, it is enough to show that if 
\[0=A+\nu v_{n,1}+\mu s_{1,1},\]
where $A$ is a linear combination of generators \eqref{eq:prev:gen}, then $\nu=\mu=0$. Let 
\[\map{\beta}{H_1(N_{g,s}^n;\zz)}{\zz_2}\]
be defined as in the previous section and define 
\[\map{\alpha}{{\Cal{M}}(N_{g,s}^n)}{\zz_2}\]
as follows: $\alpha(f)=1$ if and only if $f$ changes the local orientation around an odd number of punctures. Then there is an induced homomorphism
\[\begin{CD}H_1({\Cal{M}}(N_{g,s}^n);H_1(N_{g,s}^n;\zz)) @>(\alpha,\beta) >> H_1(\zz_2;\zz_2)\cong \zz_2 \end{CD}\]
and 
\[0=(\alpha,\beta)(A+\nu v_{n,1}+\mu s_{1,1})=\nu.\]
Now define
\[\map{\alpha'}{{\Cal{M}}(N_{g,s}^n)}{\zz_2}\]
to be the sign of the permutation
\[(z_1,\ldots,z_n)\mapsto (f(z_1),f(z_2),\ldots,f(z_n)),\]
that is $\alpha'(f)=1$ if and only if the above permutation is odd. As before, 
there is an induced homomorphism
\[\begin{CD}H_1({\Cal{M}}(N_{g,s}^n);H_1(N_{g,s}^n;\zz)) @>(\alpha',\beta) >> H_1(\zz_2;\zz_2)\cong \zz_2 \end{CD}\]
and 
\[0=(\alpha,\beta)(A+\nu v_{n,1}+\mu s_{1,1})=\mu.\]
This concludes the proof of Theorem \ref{MainThm2}.



\end{document}